\documentclass[, a4paper, 12pt]{amsart} 
\usepackage[english]{babel} 
\usepackage[utf8]{inputenc} 
\usepackage{dsfont} 
\usepackage{amssymb, amsmath, pxfonts} 
\usepackage{mathrsfs} 
\usepackage[normalem]{ulem} 
\usepackage{mathrsfs} 
\usepackage{upgreek}
\usepackage[top=3cm,left=4cm,right=4cm,bottom=5cm]{geometry} 
\usepackage{graphicx} 
\usepackage[usenames]{color} 
\makeindex 
\usepackage{amsthm}
\newtheorem{teo}{Theorem}[section]
\newtheorem{corol}{Corrollary}[section]
\newtheorem{prop}{Proposition}[section]
\newtheorem{lema}{Lemma}[section]
\newtheorem{defi}{Definition}[section]
\newtheorem{obs}{Remark}[section]
\newtheorem{ex}{Example}[section]

\usepackage{fancyhdr}
\usepackage{mathptmx}
\begin{document}
\title{A Maximum Principle for hypersurfaces in $\displaystyle \mathbb{R}^{n+1}$ with an Ideal Contact at infinity and Bounded Mean Curvature }
\author{J. Deibsom da Silva}
\author{A. F. de Sousa}

\address{Departamento de Matemática-Universidade Federal Rural de Pernambuco, 52171-900 - Recife/PE, Brazil.}
\email{jose.dsilva@ufrpe.br}
\address{Departamento de Matemática-Universidade Federal de Pernambuco, 50670-901, Recife/PE, Brazil.}
\email{tsousa@dmat.ufpe.br}

	\curraddr{}
	\urladdr{} 
    \dedicatory{}                                 
    \date{\today}                                       
    \thanks{}          
                               
    \translator{}                                 
   
    \keywords{Maximum Principles, Ideal Contact, Parabolic Manifold, Convex Side}
    \subjclass[2010]{53A10, 53A07}

\maketitle
\begin{abstract}

We will generalize a Maximum Principle at Infinity in the parabolic case given  by De Lima [Ann. Global Anal. Geom. {\bf 20}, 325-343 2001] and De Lima and Meeks [Indiana Univ. Math. Journal {\bf 53} 5, 1211–1223 2004], for disjoints hypersurfaces  of $\mathbb{R}^{n+1}$ with bounded mean curvature without restrictions on the Gaussian Curvature. We will also extend for hypersurfaces in $\mathbb{R}^{n+1}$ a generalization of Hopf's Maximum Principle for hypersurfaces that get close asymptotically.

\end{abstract}

\section{Introduction}


A classical result in Differential Geometry is the Hopf's Maximum Principle for Hypersurfaces in $\mathbb{R}^{n+1}$, which states that under certain conditions related to the Mean Curvature, if two hypersurfaces $M_1$ and $M_2$ are tangent at an interior point $p\in M_1\cap M_2$ and this point is an Ideal Contact at $p$ (see Definition \ref{contato ideal em p}), then they coincide in a neighbourhood of $p$ (Theorem \ref{hopf}).


Thinking about this type of contact between two hypersurfaces, De Lima, \cite{RFL,R-M} , and Meeks, \cite{R-M}, established an ideal contact between two disjoints surfaces $M_{1}$ and $M_{2}$ in $\mathbb{R}^{3}$, which generalizes the \emph{Ideal Contact at $p$} for disjoints surfaces that get asymptotically close to each other. This approximation was name {\it Ideal Contact at infinity} (see Definition \ref{contato}).


Assuming an \emph{Ideal Contact at Infinity} between the surfaces $M_{1}$ and $M_{2}$ in $\mathbb{R}^{3}$, De Lima demonstrates the following {\it Maximum Principle at Infinity} for surfaces with bounded Gaussian curvature and constant Mean Curvature $H\neq 0$, \cite{RFL}.



\begin{teo}\label{H-superfices em R3}
Let $M_1$ and $M_2$ be two disjoints, complete and properly embedded $H$-surfaces in $\mathbb{R}^3$, with bounded Gaussian Curvature and non-empty boundaries $\partial M_{1}$ and $\partial M_{2}$. If $M_{1}$ and $M_{2}$ have an ideal contact at infinity and either $M_{1}$ or $M_{2}$ is parabolic, then
$$\min\{dist(M_{1},\partial M_{2}),dist(M_{2},\partial M_{1})\}=0.$$

\end{teo}

In 2004, De Lima, \cite{RFL,R-M}, along with Meeks managed to prove in \cite{R-M} the following non-parabolic version with bounded Gaussian and mean curvatures of Theorem \ref{H-superfices em R3}.



\begin{teo}\label{Meeks-Lima}
Let $M_1$ be a surface with boundary $\partial M_1$ and bounded Gaussian curvature, which is properly embedded in $\mathbb{R}^3$ and whose mean curvature satisfies $b_{0}\leq H_{M_1}\leq b_{1}$, $b_{0},b_{1}>0$. Assume $M_2$ is a surface with boundary $\partial M_2$, which is properly immersed in $\mathbb{R}^3$ and such that $\vert H_{M_2}\vert\leq b_0$. Then, if $M_2$ has an contact ideal at infinity wich $M_1$, one has
$$\min\{dist(M_1,\partial M_2),dist(M_2,\partial M_1)\}=0.$$
\end{teo}


As an application of Theorem \ref{Meeks-Lima} De Lima and Meeks also proved the following theorem, which generalizes Hopf's Maximum Principle for surfaces in $\mathbb{R}^{3}$ with an ideal contact at infinity and bounded Gaussian and mean curvatures:


\begin{teo}\label{convexo em R3}

Suppose $M_{1}$ is a properly embedded surface in $\mathbb{R}^3$ without boundary and of bounded Gaussian Curvature. If the mean curvature function of $M_{1}$ satifies $b_{0}\leq H_{M_1}\leq b_{1}$, $b_{0},b_{1}>0$, the surface $M_{2}$ without boundary, which is properly immersed in $\mathbb{R}^3$ and whose mean curvature satisfies $\vert H_{M_2}\vert\leq b_{0}$, cannot lie on the mean convex side of $M_{1}$.
\end{teo}
In \cite{RFL} De Lima proved the parabolic version of Theorem \ref{convexo em R3} assuming that $M_{1}$ and $M_{2}$ are $H$-surfaces, $H\neq 0$.


Our objective in this article is to extend to hypersurfaces of $\mathbb{R}^{n+1}$ Theorems \ref{H-superfices em R3} and \ref{Meeks-Lima} above. It will be done in Theorem \ref{principio maximo} where we prove the Maximum Principle at Infinity for properly embedded and disjoints hypersurfaces $M_{1}$ and $M_{2}$ in  $\mathbb{R}^{n+1}$ with nonempty boundaries. To do that, we will suppose that $M_{2}$ is complete and that  $\sup _{M_2}\vert{\bf H}_{M_2}\vert \leq\inf _{M_1} H_{M_1}$, where $H_{M_1}$ is the mean curvature of $M_{1}$ and ${\bf H}_{M_2}$ is the mean curvature vector of $M_2$. We will also assume that $M_{2}$ have an ideal contact at infinity with $M_1$ and that $M_{2}$ is parabolic. However, we will not consider any additional hypothesis about the Gaussian curvature of any hypersurface. We will need two lemmas that will be proved in section \ref{secprincipio}, lemmas \ref{lema2.1} and \ref{lema2.2}.


As an application of Theorem \ref{principio maximo} we will prove Theorem \ref{convexo}, which is a generalization of Hopf's Maximum Principle for hypersurfaces with an ideal contact, which extends Theorem \ref{convexo em R3}. Such theorem states that under certain conditions, like the ones in Theorem \ref{principio maximo}, if $M_{2}$ has an empty boundary then it cannot be on the convex side of $M_{1}$. In Theorem \ref{convexo1.2} we will extend Theorem \ref{convexo} to the case where $M_{1} \subset \mathbb{R}^{n+k}$ is a hypersurface and $M_{2} \subset \mathbb{R}^{n+k}$ a parabolic $n$-submanifold, which generalizes Theorem 4 in \cite{U-S} for asymptotic hypersurfaces.




In Theorem \ref{dist}, as in its corollaries, we used the Omory-Yau's Maximum Principle, see \cite{omori,yau}, and prove an analogous result to Theorem \ref{convexo} without the hypothesis of $M_2\subset\mathbb{R}^{n+1}$ be a parabolic hypersurface, obtaining another generalization of Hopf's Maximum Principle for asymptotics hypersurfaces of $\mathbb{R}^{n+1}$.

\section{Preliminaries}


Given a smooth and oriented hypersurface $M\subset \mathbb{R}^{n+1}$, we denote the mean curvature function and the mean curvature vector of $M$ as $H_{M}$ and $\textbf{H}_{M}$, respectively. We will also denote by $\nabla^{M}$ and $\Delta^{M}$ the gradient and Laplacian of $M$, respectively.

\subsection{Parabolic Riemannian manifold }

Let $M^{n}$ be an $n$-dimensional Riemannian manifold  with smooth (possibly empty) boundary $\partial M$ and $\Omega\subset M$ an open set of $M$. A function $h\in C^2(\Omega)$ is said \emph{subharmonic} if 
$$\Delta ^Mh\geq 0.$$

Subharmonic functions will play an important role in this section in
that we will address a class of Riemannian manifolds that are characterized
by these functions.


When a Riemannian manifold $M$ has an empty boundary $\partial M$ we will say that $M$ is parabolic if it cannot exist a non-constant upper bounded subharmonic function, namely, $M$ is parabolic if $\Delta ^Mh\geq 0$ and $\sup _Mh<+\infty$ we have $h$ cosntant. As an example of such Riemannian manifolds Cheng e Yau showed that $M=(\mathbb{R}^2,\langle ,\rangle _{can})$ is a parabolic Riemannian manifold, \cite{cheng-yau}.  In the case when $\partial M$ is nonempty will use the following definition of parabolic manifold given by De Lima in \cite{RFL}.


\begin{defi}\label{defi1}
An $n$-dimensional complete Riemannian manifold $M^{n}$ with nonempty smooth boundary $\partial M$ is called \textbf{parabolic} if for any upper bounded subharmonic function $h$ in $M$,  we have
$$\sup_{M}h=\sup_{\partial M}h.$$
\end{defi}


The next result, whose proof can be found in \cite{RFL}, gives a sufficient condition for a Riemannian manifold with nonempty boundary be parabolic.

\begin{prop}[Proposition 2 of \cite{RFL}]\label{harmparb}
Let $M$ be a complete Riemannian manifold with nonempty, smooth boundary $\partial M$. If there exist a proper, positive harmonic funcition defined on $M$, then $M$ is parabolic.
\end{prop}

\begin{ex}\label{cilindro}
Let $C$ be a cylinder in $\mathbb{R}^3$ given by $C=\{(x_1,x_2,x_3):x_1^2+x_2^2=1,\;x_3\geq 1\}$. Consider the parametrization  given by $X(\theta ,r)=(\cos\theta ,\sin\theta ,r)$, $ 0\leq\theta < 2\pi$, $r\geq 1$. It is easy to see that the function $h(\theta ,r)=r$ on $C$ is positive, proper and harmonic. Thus, Proposition \ref{harmparb} gives $C$ parabolic.
\end{ex}



Next, we will state a proposition of fundamental importance in the demonstration of our main result, whose proof can be found in \cite{RFL}. In the following, $M^{n}$ is a complete, $n$-dimensional Riemannian manifold,  with a (possibly empty) smooth boundary $\partial M$ and $M'\subset M$ a complete and embedded $n$-dimensional Riemannian submanifold of $M$  with nonempty smooth boundary $\partial M'$.



\begin{prop}[Prposition 3 of \cite{RFL}]\label{prop1}
Let $M$ and $M'$ be like above. Then $M'$ is parabolic if $M$ is parabolic.
\end{prop}
We will also use the following lemma, found in \cite{L-T}.


\begin{lema}[Lemma 2.3 of \cite{L-T}]\label{lema1.1}
Let $A$ be a quadratic form in an $n$-dimensio\-nal Euclidean vectorial space with eigenvalues $\lambda_ {1} \leq \cdots \leq \lambda _{k} \leq \cdots \leq \lambda_{n}$. Then for any $k$-dimensional subspace $W \subset V$ we have 
$$\mathrm{tr} A \mid_{W}\geq \lambda_{1} + \cdots + \lambda_{k} .$$
\end{lema}


Comparing to the maximum principle, we will define next the ideal contact at a point (Definition \ref{contato ideal em p}) and show Hopf's Maximum Principle (Theorem \ref{hopf}) which, inspired the concept of an ideal contact at infinity (see Definition \ref{contato}) for hypersurfaces, \cite{RFL,R-M}.


\begin{defi}[Ideal Contact at $p$]\label{contato ideal em p}
Let $M_{1}$ and $M_{2}$ be two oriented hypersurfaces in $\mathbb{R}^{n+1}$. If $M_{1}$ and $M_{2}$ are tangent at an interior point $p$ and has the same unit normal $\eta_{0}$ at $p$, we will say that they have an \textbf{Ideal contact at $p$}. We also say that $M_{1}$ lies above $M_{2}$ near $p$ with respect to $\eta_{0}$, if when we express $M_{1}$ and $M_{2}$ as graphics of function $\phi_{1}$ and $\phi_{2}$ over the tangent hyperplan in $p$ we have $\phi_{1} \ge \phi_{2}$ in a neighbourghood of $p$.
\end{defi}


\begin{teo}[Hopf's Maximum Principle, \cite{tangencia}]\label{hopf}
Let $M_1$ and $M_2$ be oriented hypersurfaces in $\mathbb{R}^{n+1}$ which have a contact at a point p. Let $H_{M_1}$ and $H_{M_2}$ be their mean curvature function, respectively. If $H_{M_1} \leq H_{M_2}$ at $p$ then $M_1$ cannot lie above $M_2$, unless they coincide in a neighborhood of $p$.
\end{teo}


\section{The Maximum Principle at Infinity for hypersurfaces in $\mathbb{R}^{n+1}$}\label{secprincipio}


We introduze now, and we will use it in the course of this work, the definition of {\it Ideal Contact at Infinity} used by De Lima and Meeks, \cite{R-M}.

\begin{defi}[Ideal contact at infinity]\label{contato}
Let $M_{1}$ be a propperly embedded hypersurface  in $\mathbb{R}^{n+k}$ with a positive mean curvature function. We say that an $n$-dimensional submanifold $M_{2} \subset\mathbb{R}^{n+k}$ has an \textit{Ideal Contact at Infinity} with $M_{1}$ if $M_{1}$ and $M_{2}$ are disjoints and there exist sequences of interior points $y_i\in M_1$,  $x_i\in M_2$ and $\lambda _i >0$, $i\in \mathbb{N}$, with
\begin{center}
$\vert y_i - x_i\vert\rightarrow 0$ \ \ \ \ \ \ and\ \ \ \ \ \  $x_i - y_i = \lambda _i{\bf H}_{M_1}(y_i)$
\end{center}
always when $i\rightarrow +\infty$. , Figure \ref{deibsom}.
\end{defi}

Here, as in \cite{R-M}, we say that two disjoints and properly imersed hypersurfaces $M_1$ and $M_2$, with nonempty boundaries $\partial M_1$  and $\partial M_2$  satisfy the {\it Maximum Principles at Infinity} if
$$dist(M_1,M_2)=\min\{dist (M_1,\partial M_2), dist(M_2,\partial M_1)\},$$
were $dist$ is the distance in $\mathbb{R}^{n+1}$.

\begin{figure}[!htb]
  
    \centering
    
    \includegraphics[scale =.35]{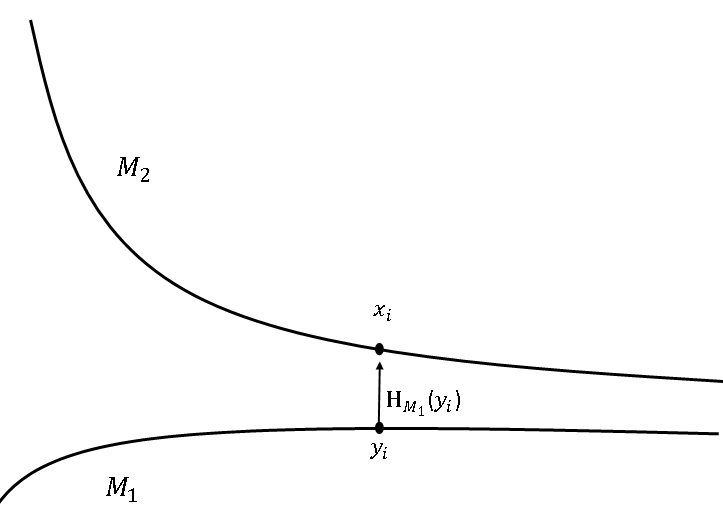}
    
   \caption{Ideal Contact at Infinity}
    
    \label{deibsom}
\end{figure}


In this section we state the main result of this paper. Here we suppose $M_{1} \subset \mathbb{R}^{n+1}$ to be an oriented smooth hypersurface with positive mean curvature $H_{M_{1}}$.

 
\begin{teo}[Maximum Principles at Infinity]\label{principio maximo}
Let $M_{1}$ and $M_{2}$ two propperly embedded and disjoints hypersurfaces in $\mathbb{R}^{n+1}$ with  nonempty boundaries $\partial M_{1}$ and $\partial M_{2}$. Suppose that $M_{2}$ is complete and that
 \begin{equation}
    \sup _{M_2}\vert\mathbf{H}_{M_2}\vert \leq b_0\leq\inf _{M_1}H_{M_1}\label{hip1},\ \ b_0>0.
 \end{equation}
 If $M_{2}$ has an ideal contact at infinity with $M_{1}$ and $M_2$ is parabolic, then
 \begin{equation}
 \min\{dist(M_1,\partial M_2),dist(M_2,\partial M_1)\}=0.\label{tese}
 \end{equation}
\end{teo} 
 

Such theorem is a generalization of Theorems \ref{H-superfices em R3} and \ref{Meeks-Lima}. Its demonstration leads us to Theorem \ref{convexo} which is a generalization of Hopf's Maximum Principle in $\mathbb{R}^{n+1}$ for hypersurfaces with an ideal contact at infinity (see Definition \ref{contato}) in a way that they are disjoints hypersurfaces that approach each other asymptotically. Exemple \ref{contracii} bellow shows that Theorem \ref{principio maximo} may be false without the hypothesis of {\it Ideal Contact at Infinity}.

\begin{ex}\label{contracii}
Let $M_1$ be the surface of revolution obtained by rotating the curve $\alpha (t)=(t,0,\dfrac{1}{1-t^2})$, $0<t_0< t<1$, about the $z$ axis and $M_2$ the cylinder $M_2=\{(x,y,z):x^2+y^2=1,\; z>z_0>0$\}. By Exemple \ref{cilindro} $M_2$ is parabolic, is disjoint of $M_1$ and we have $\sup _{M_2} \vert H_{M_2}\vert =\dfrac{1}{2}$. We also have 
$\inf _{M_1}H_{M_1}=\dfrac{1}{2}$, but $$0=dist(M_1,M_1)<\min\{dist(M_1,\partial M_2),dist(M_2,\partial M_1)\}.$$
This happens because there is no $y\in M_1$, $x\in M_2$ and $\lambda >0$ such that, $x-y=\lambda \mathbf{H}_{M_1}(y)$, that is, $M_2$ does not have an ideal contact at infinity with $M_1$.

\end{ex}

\subsection{Preliminary results}

In the demonstration of Theorem \ref{principio maximo} we will use Lemmas \ref{lema2.1} and \ref{lema2.2}, that demand the following assumptions: Let $M_{1} \subset \mathbb{R}^{n+k}$ be a hypersurface, at least $C^{2}$, with mean curvature $H_{M_{1}}$ according to the unit normal $\eta$. We denote by
$$\Uplambda _n:=\dfrac{1}{n}(\lambda _1 + \cdots + \lambda _n)$$
the $n$-th mean curvature with respect to $\eta$, where $\lambda_{1} \leq \lambda_{2}\leq\cdots\leq \lambda_{n+k-1} $ are the principal curvatures of $M_{1}$ with respect to $\eta$. Also let $M_2\subset \mathbb{R}^{n+k}$ be a $n$-dimensional $C^{2}$ submanifold, $n\geq 1$, with mean curvature vector $${\bf H}_{M_2}=-\dfrac{1}{n}\sum_{r=1}^{k}(\textmd{div}\;\eta ^r)\eta ^r,$$ where $\eta^{1}, \cdots, \eta^{k}$ are orthonormal vector fields normal to $M_{2}$.


Let $d$ be the distance function $d(x):=dist(x,M_{1})$. Such function is of class $C^{2}$ in a neighborhood of $M_{1}$, Lipschitz with constant 1 and oriented by the choice of $\eta$, i.e., $\eta(y)=Dd(x)$, where $y\in M_{1}$ is such that $\vert x-y\vert=d(x)$ and $D=\left(\dfrac{\partial}{\partial x^1}, \dots , \dfrac{\partial}{\partial x^{n+k}}\right)$ is the gradient of $\mathbb{R}^{n+k}$, see \cite{U-S,G-T} . This way, the point $x$ is such as $x=y+d(x)\eta(y)$.


For each $x$ close to $M_{1}$ we consider a parallel hypersurface
$$M_{1d(x)}=\{p\in \mathbb{R}^{n+k}:d(p)=d(x)\}=d^{-1}(d(x)).$$
Such hypersurfaces are of class $C^{2}$ and have principal curvatures at $x_{0}$ given by
$$\dfrac{\lambda _1(y_0)}{1-\lambda _1d(x_0)}\leq \dfrac{\lambda _2(y_0)}{1-\lambda _2d(x_0)}\leq\cdots\leq \dfrac{\lambda _{n+k-1}(y_0)}{1-\lambda _{n+k-1}d(x_0)},$$
where $y_{0} \in M_{1}$ is such that $\vert x_{0} - y_{0}\vert = d(x_0)$ and $\lambda _1(y_0)\leq\lambda _2(y_0)\leq\cdots\leq\lambda _{n+k-1}(y_0)$ are the principal curvatures of $M_{1}$ at $y_{0}$ if $\vert d\vert \ll1$, \cite{U-S,G-T}. Observe that this give us
$$\dfrac{1}{n}\left(\dfrac{\lambda _1(y_0)}{1-\lambda _1d(x_0)}+ \dfrac{\lambda _2(y_0)}{1-\lambda _2d(x_0)}+\cdots + \dfrac{\lambda _{n}(y_0)}{1-\lambda _{n}d(x_0)}\right)\geq \dfrac{1}{n}(\lambda _1 + \cdots + \lambda _n)$$ for any hypersurface $M_{1d}$. In other words, the $n$-mean curvature $\Uplambda_{n}(x_{0})$ of $M_{1d}$ at $x_{0}$ is not smaller than the $n$-mean curvature of $M_1$ at $y_0$.


In the following let $\{e_{1}, e_{2}, \cdots , e_{n}\}$ be an orthonormal basis of $T_xM_2$ and denote by $$e_i^{T}:=e_i-\langle e_i, \eta \rangle \eta$$ the orthogonal projection of $e_{i}$ over the tangent space $T_xM_{1d}$. Also let $T_xM_2^{T}$ be the orthogonal projection space of $T_xM_2$ over $T_xM_{1d}$. Finally, let $II_d$ and $A_d$ be the second fundamental form and the shape operator of $M_{1d}$, with respect to $\eta$, respectively.


The following lemma is an adaptation of Lemma 1 in \cite{U-S}, with an analogue demonstration.


\begin{lema}\label{lema2.1}
Let $M_{1}$ and $M_{2}$ be like above and $d$ be the distance function $d(x)=dist(x,M_{1})$. Suppose in addition that $M_{2}$ has an ideal contact at infinity with $M_{1}$. So, we have
$$\Delta ^{M_2}d - \displaystyle\sum_{i,j}^{n}\upsigma_{ij}II_d(e_i^T,e_j^T) - n\langle {\bf H}_{M_2},Dd\rangle + \mathrm{tr}A_d\mid_{T_xM_2^T}=0,$$
where $\upsigma_{ij}=\dfrac{\nabla _{e_i}^{M_2}d\nabla _{e_j}^{M_2}d}{1-\vert\nabla^{M_2}d\vert ^2}$ and $\nabla_{e_i}^{M_2}$ is the derivative in the direction of $e_{i}$.
\end{lema}

\begin{proof}
Let $x \in M_{2}$ be such that $d(x)\ll1$. Such point exist because $M_{2}$ has an ideal contact at infinity with $M_{1}$. Then, if $\eta^{1}, \cdots,\eta^{k}$ form an orthonormal basis of $T_xM_2^{\perp}$ and $Dd=\eta$ is the Euclidian gradient of $d$, we have that
$$\nabla ^{M_2}d= Dd - \langle Dd,\eta ^1\rangle\eta ^1 - \cdots - \langle Dd,\eta ^k\rangle\eta ^k$$ and
\begin{equation}
\begin{array}{rlll}
\Delta ^{M_2}d & =& \textmd{div}\nabla ^{M_2}d&= \textmd{div}\; Dd - \displaystyle\sum_{r=1}^{k} \langle Dd,\eta ^r\rangle \textmd{div}\;\eta ^r\\
& & & = \textmd{div}\; Dd + n\langle{\bf H}_{M_2},Dd^{\perp}\rangle
\end{array}
\label{eq1}
\end{equation}
where $Dd^{\perp}=\displaystyle\sum _{r=1}^k\langle Dd,\eta ^r\rangle\eta ^r$ is the normal component of $\eta=Dd$ relative to $M_{2}$ and ${\bf H}_{M_2}=-\displaystyle\dfrac{1}{n}\sum_{r=1}^{k}(\textmd{div}\;\eta^r)\eta^r$ is the mean curvature vector of $M_{2}$.

Denote by $\overline{\nabla}_{e_{i}}$ the Euclidian directional derivative in the direction of $e_{i}$. Then, as 
\begin{center}
$e_i=e_i^T + \langle e_i,\eta\rangle\eta$\ \ and\ \ $\vert\eta\vert ^2=1,$
\end{center}
it follows from (\ref{eq1}) that
$$\begin{array}{rll}
\Delta ^{M_2}d &=&\displaystyle \sum_{i=1}^{n}\langle e_i, \overline{\nabla}_{e_i}\eta\rangle + n\langle {\bf H}_{M_2},Dd^{\perp}\rangle\\
&=& \displaystyle \sum_{i=1}^{n}\langle e_i^T, \overline{\nabla}_{e_i^T}\eta\rangle  + n\langle {\bf H}_{M_2},Dd^{\perp}\rangle\\
&=& -\displaystyle\sum_{i=1}^{n}II_d( e_i^T, e_i^T) + n\langle {\bf H}_{M_2},Dd^{\perp}\rangle
\end{array}$$
and as $Dd=\nabla^{M_2}d + Dd^{\perp}$, we have that
\begin{equation}
\Delta ^{M_2}d - n\langle {\bf H}_{M_2},Dd\rangle + \displaystyle\sum_{i=1}^{n}II_d( e_i^T, e_i^T)=0. \label{eq2}
\end{equation}



As $d(x)\ll 1$, because $M_2$ have a  {\it contact ideal at infinity} with $M_1$, we have that $\vert \nabla^{M_2}d\vert (x)< 1$. That is, $\vert \nabla^{M_2}d\vert (x)< 1$ if $dist(x,M_1)$ is smoll enough. Then, if we put

$$\begin{array}{rllll}
g_{ij}&:=&\langle e_i^T, e_j^T\rangle =\langle e_i - \langle e_i,\eta\rangle\eta,e_j - \langle e_j,\eta\rangle\eta\rangle\\
&=&\delta_{ij}-\langle e_i,\eta\rangle\langle e_j,\eta\rangle
\end{array}$$
and as $\vert\nabla ^{M_2}d\vert <1$, we have that $g^{ij}$, the inverse  of $g_{ij}$ is given by
$$g^{ij}=\delta _{ij} + \dfrac{\langle e_i,\eta\rangle\langle e_j,\eta\rangle}{1-\sum_{i=1}^{n}\langle e_i, \eta\rangle ^2}=:\delta _{ij}+\upsigma _{ij}.$$
And then, the trace of $A_d$ in $T_x{M_2}^T$ is given by

$$
\begin{array}{rll}
\mbox{tr}A_d\mid_{T_x{M_2}^T}&=&\displaystyle\sum_{i,j}^{n}g^{ij}\displaystyle II_d( e_i^T, e_j^T)\\
&=&\displaystyle\sum_{i=1}^{n}II_d( e_i^T, e_i^T) + \displaystyle\sum_{i,j}^{n}\upsigma _{ij}II_d( e_i^T, e_j^T).
\end{array}
$$
From (\ref{eq2}) we have

\begin{equation}
\Delta ^{M_2}d - n\langle{\bf H}_{M_2},Dd\rangle + \mbox{tr}A_d\mid_{T_x{M_2}^T} - \displaystyle\sum_{i,j}^{n}\upsigma _{ij}II_d( e_i^T, e_j^T)=0. \label{eq3}
\end{equation}
And the lemma goes on observing that $$\langle e_i,\eta\rangle = \langle e_i,Dd\rangle =\langle e_i,\nabla ^{M_2}d\rangle + \langle e_i,Dd^{\perp}\rangle =  \langle e_i,\nabla ^{M_2}d\rangle = \nabla _{e_i}^{M_2}d$$ and $$\nabla ^{M_2}d =\sum_{i=1}^{n}(\nabla _{e_i}^{M_2}d)e_i.$$ Then
$$\upsigma _{ij}= \displaystyle\dfrac{\langle e_i,\eta\rangle\langle e_j,\eta\rangle}{1-\sum_{i=1}^{n}\langle e_i, \eta\rangle ^2}= \dfrac{\nabla _{e_i}^{M_2}d\nabla _{e_j}^{M_2}d}{1-\vert\nabla^{M_2}d\vert ^2} .$$ 
\end{proof}




\begin{lema}\label{lema2.2}
Let $M_{1}$ and $M_{2}$ be like in Lemma \ref{lema2.1} and $d(x)=dist(x,M_{1})$. Suppose that $M_{2}$ has an ideal contact at infinity with $M_{1}$ and that
$$\sup _{M_2}\vert{\bf H}_{M_2}\vert \leq\inf _{M_1}\Uplambda _n.$$
Then we have that
$$\Delta ^{M_2}d -C_0 \vert\nabla ^{M_2}d\vert ^2\leq 0,$$
for some positive constant $C_0$.
\end{lema}


\begin{proof}
First we will prove that
\begin{equation}
- n\langle {\bf H}_{M_2},Dd\rangle + \mbox{tr}A_d\mid_{T_xM_2^T}\geq 0. \label{eq4}
\end{equation}
By Lemma \ref{lema1.1} we have that
$$
\begin{array}{rll}
\dfrac{1}{n}\mbox{tr}A_d\mid_{T_xM_2^T}&\geq&\dfrac{1}{n}\left(\dfrac{\lambda _1(y)}{1-\lambda _1d(x)}+ \dfrac{\lambda _2(y)}{1-\lambda _2d(x)}+\cdots + \dfrac{\lambda _{n}(y)}{1-\lambda _{n}d(x)}\right)\\
&\geq& \dfrac{1}{n}(\lambda _1(y) + \cdots + \lambda _n(y))
\end{array}
$$
where $y \in M_{1}$ is such that $\vert x-y\vert = d(x)\ll 1$ and $\lambda _1(y), \cdots , \lambda _n(y)$ are the $n$ first principal curvatures  of $M_{1}$. As we supposed that
$$\sup _{M_2}\vert{\bf H}_{M_2}\vert \leq\inf _{M_1}\Uplambda _n, $$
we conclude that $\dfrac{1}{n}\mbox{tr}A_d\mid_{T_xM_2^T}\geq \vert{\bf H}_{M_2}\vert$ and by the Schwarz inequality, it follows that $$\langle {\bf H}_{M_2}, Dd\rangle \leq \vert{\bf H}_{M_2}\vert\vert Dd\vert = \vert{\bf H}_{M_2}\vert ,$$
since $Dd =\eta$. From this inequality we derive (\ref{eq4}).


For the sake of simplicity of notation denote $\nabla_{e_i}^{M_2}d=d_i$. Let $\upsigma_{ij}$ be given by Lemma \ref{lema2.1}, then
\begin{equation}
\begin{array}{rll}
\displaystyle\sum_{i,j}^{n}\upsigma _{ij}II_d( e_i^T, e_j^T)&=&\displaystyle\sum_{i,j}^{n}\dfrac{II_d(e_i^T, e_j^T)d_id_j}{1-\vert\nabla^{M_2}d\vert ^2}\\
 &=&\displaystyle\sum_{i,j}^{n}\dfrac{II_d( d_ie_i^T, d_je_j^T)}{1-\vert\nabla^{M_2}d\vert ^2}\\
 &=&\displaystyle\dfrac{II_d((\nabla ^{M_2}d)^T,(\nabla ^{M_2}d)^T)}{1-\vert\nabla^{M_2}d\vert ^2}\\
 &\leq& C_0\vert \nabla ^{M_2}d\vert ^2, \label{eq6}
\end{array}
\end{equation}
for a positive constant $C_{0}$. From $\vert \nabla ^{M_2}d\vert \ll 1$ we conclude the inequality of (\ref{eq6}), because $M_{2}$ has an ideal contact at infinity with $M_{1}$. Using now (\ref{eq3}), (\ref{eq4}) and (\ref{eq6}), we conclude the lemma.

\end{proof}

\subsection{Proof of the main theorem}
The demonstration of Theorem \ref{principio maximo} is similar to the demonstration in the case where we have $H$-surfaces in $\mathbb{R}^{3}$ given in Theorem 1 in \cite{RFL}. The difference between them lies in the demonstrations  of Lemmas \ref{lema2.1} and \ref{lema2.2} above, which equivalents in \cite{RFL} are the Lemmas 3 and 4, respectively. These will allow us to construct in a convenient set  a subharmonic function and assuming by absurd that (\ref{tese}) is not true we reach a contradiction related to the parabolicity of $M_{2}$.


\begin{proof}[\hypertarget{PMI}{Proof of Theorem} \ref{principio maximo}]
Let's suppose that (\ref{tese}) is false, i.e., 
$$m_0=\min\{dist(M_1,\partial M_2),dist(M_2,\partial M_1)\}>0.$$
For $\epsilon >0$ sufficient small, let  $$M_2(\epsilon)=\{x\in M_2 : dist(x,M_1)\leq \epsilon\}.$$ For each $x\in M_2(\epsilon)$, consider the set $$S_x=\{y\in M_1: \vert y - x\vert=dist(x,M_1) \ \ and \ \  x-y=\lambda{\bf H}_{M_1}(y), \ \ \lambda> 0\}$$ and finally define the set $M'_2(\epsilon)\subset M_2(\epsilon)$ as $$M'_2(\epsilon)= \{x\in M_2(\epsilon): S_x\neq \emptyset\}.$$ Note that $M'_2(\epsilon)$ is non-empty because $M_{1}$ and $M_{2}$ are propperly embedded in $\mathbb{R}^{n+1}$ and have an ideal contact at infinity. Let $C_2(\epsilon)\subset M'_2(\epsilon)$ be a conex component of $M'_2(\epsilon)$. Now take $\epsilon >0$ such that $m_o>\epsilon$. From that last assumption we have that $\partial M_2\cap C_2(\epsilon)=\emptyset$. In fact, if otherwise we had $x\in C_2(\epsilon)\subset M'_2(\epsilon)$ we would have $dist(x,M_1)\leq\epsilon$ and $x\in \partial M_2$ we would have $dist(x,M_1)>\epsilon$, because $dist(M_1,\partial M_2)>\epsilon$, which give us $$\partial C_2(\epsilon)= \{x\in C_2(\epsilon): dist(x,M_1)=\epsilon\}.$$ Consider now the distance function $d(x)=dist(x,M_1)$. By the lemma \ref{lema2.2} we have that


\begin{equation}
\Delta ^{M_2}d - C_0\vert\nabla ^{M_2}d\vert ^2\leq 0 \label{eq1teo}
\end{equation}
for a positive constant $C_{0}$. Observe that $d\mid _{\partial C_2(\epsilon)}\equiv\epsilon$. We also have that $C_2(\epsilon)$ is not compact, otherwise we would have a $x'$ in the interior of $C_{2}(\epsilon)$ such that $d(x')$ would be minimum, and in this case $\nabla ^{M_2}d(x')=0$ and by (\ref{eq1teo}) we would have  $\Delta ^{M_2}d(x')\leq 0$, contrary to the fact that $x'$ is a inferior minimum point. So $C_2(\epsilon)$ is not compact and $\sup_{C_2(\epsilon)}d=\epsilon$.

Consider now a function $\phi$ in $C_2(\epsilon)$ given by $$\phi (x)=e^{-C_0d(x)}.$$ Calculating $\Delta ^{M_2}\phi$ using (\ref{eq1teo}), we will have that $$\Delta ^{M_2}\phi = -C_0e^{-C_0d}(\Delta ^{M_2}d - C_0\vert\nabla ^{M_2}d\vert ^2)\geq 0$$  from which we can conclude that $\phi$ is subharmonic in $C_{2}(\epsilon)$. As we are assuming that $M_{2}$ is parabolic, we have by the Proposition \ref{prop1} that $C_2(\epsilon)$ is parabolic. So we should have $$\sup _{C_2(\epsilon)}\phi = \sup _{\partial C_2(\epsilon)}\phi = e^{-C_0 \epsilon}$$ which is a contradition because $\sup _{C_2(\epsilon)}\phi = 1 > e^{-C_0 \epsilon}=\sup _{\partial C_2(\epsilon)}\phi$. As this contradition came from the assumption that  $m_0>0$, we have that  $$\min\{dist(M_1,\partial M_2),dist(M_2,\partial M_1)\}=0$$ and the proof is complete.
 \end{proof}
 
In \cite{pigola-setti} Impera-Pigola-Setti they use the following definition of parabolic Riemannian manifol when $\partial M\neq \emptyset$ (see also \cite{impera-pigola-setti}).
 
 \begin{defi}\label{defi1.1}
 Let $M$ an oriented Riemannian manifold with smooth boundary $\partial M\neq \emptyset$  and exterior unit normal $\nu$. $M$ is said to be $\mathcal{N}$-parabolic if the only solutions to the problem
 \begin{equation}\label{neuman}
\left\{
  \begin{array}{lll}
  \Delta ^{M}h\geq 0 & em& M\\
  \dfrac{\partial h}{\partial \nu}\leq 0 & em& \partial M\\
  \sup _{M}h<+\infty
 \end{array}
  \right.
  \end{equation}
 is the constant function $h\equiv\sup _{M}h$.
 \end{defi}
 
 Getting the following proposition (Appendix A in \cite{pigola-setti}).
 
 \begin{prop}\label{Nparabol}
 Assume that $M$ is an $\mathcal{N}$-parabolic manifold with boundary
$\partial M \neq \emptyset$ and let $h$ be a solution of the problem
$$\left\{
  \begin{array}{lll}
  \Delta ^{M}h\geq 0 & em& M\\
  \sup _{M}h<+\infty .
 \end{array}
  \right.$$
  Then $$\sup _Mh=\sup _{\partial M}h.$$
 \end{prop}
 
 Proving thus that Definition \ref{defi1.1} implies the Definiton \ref{defi1} given by De Lima em \cite{RFL}. Naturally if $\partial M = \emptyset$, then the $\mathcal{N}$-parabolicity is equivalent the parabolicity.
 
 In \cite{A. Grigoryan,Grigoryan2} Grigor'yan proved the following theorem
 
 \begin{teo}\label{grigoryan}
 Let $M$ a complete Riamannian manifold. If for some point $o\in M$
 $$\dfrac{R}{Vol B_R^M(o)}\notin L^1(+\infty)$$
or
$$\dfrac{1}{Area(\partial _0 B_R^M(o))}\notin L^1(+\infty)$$
then $M$ is $\mathcal{N}$-parabolic.
 \end{teo}
 
 Were, following the notation of \cite{pigola-setti,impera-pigola-setti} for a non-necessarily connected open set $\Omega\subseteq M$, we defined
 $$\partial _0\Omega=\partial \Omega\cap int M$$
and
$$\partial _1\Omega=\partial M\cap \Omega.$$
Which gives us 
$$\partial _0 B_R^M(o)=\partial B_R^M(o)\cap int M.$$
 
 The Theorem \ref{grigoryan} has as corollary the next result proved by Cheng and Yau, telling us that $(\mathbb{R}^2,\langle ,\rangle _{can})$ is an $\mathcal{N}$-parabolic Riemannian manifold, \cite{cheng-yau}.
 
 \begin{corol}[Cheng-Yau]
 Let $M$ a complete Riemannian manifold. If for some point $o\in M$ and for some sequence $R_k\to +\infty$
 $$VolB^M_{R_k}(o)\leq cte.R^2_k,$$
 then $M$ is $\mathcal{N}$-parabolic.
 
 \end{corol}
 
 
 Observin now that in the \hyperlink{PMI}{Proof} of Theorem \ref{principio maximo}, we have that $\partial _0C_2(\epsilon)=\partial C_2(\epsilon)$. Therefore, the Ahlfors Maximum Principles, Theorem 7 in \cite{pigola-setti} (see also \cite{impera-pigola-setti}), tells us that if $M_2$ is $\mathcal{N}$-parabolic in Theorem \ref{principio maximo}, that is, does not admit non constant function satisfying (\ref{neuman}), then the function $\phi (x)=e^{-C_0d(x)}$ that satisfies $\Delta ^{M_2}\phi\geq 0$ on $C_2(\epsilon)$ is such that
 $$\sup _{C_2(\epsilon)}\phi = \sup _{\partial _0 C_2(\epsilon)}\phi.$$
 Allowing us to reach the same contradiction \hyperlink{PMI}{Proof} of Theorem \ref{principio maximo}. Guaranteeing the validity of {\it Maximum Principles at Infinity} for $\mathcal{N}$-parabolic Riemannian manifold.
 
 In \cite{leandro}, see also Appendix A in \cite{pigola-setti}, Pessoa-Pigola-Setti they extend the notion of $\mathcal{N}$-parabolic Riemannian manifold with the next definition
 
 \begin{defi}
 We say that a Riemannian manifold $M$ with nonempty boundary $\partial M$ is $\mathcal{D}$-parabolic if every bounded function $h\in C^{\infty}(int\; M)\cap C^0(M)$ satisfiyng
 $$\left\{
  \begin{array}{lll}
  \Delta ^{M}h= 0 & em& int\;M\\
  h=0&em&\partial M,
 \end{array}
  \right.$$
  vanishes identically.
 \end{defi}
 
 Not that of Proposition \ref{Nparabol} every $\mathcal{N}$-parabolic Riemannian manifold is $\mathcal{D}$-parabolic, but the converse is not true, see Example 4 of \cite{leandro}.
 
 When $M$ is a $\mathcal{D}$-parabolic Riemannian manifold we have the following proposition.
 
 \begin{prop}[Proposition 10 of \cite{leandro}]\label{Dparabolica}
 Let $M$ be a manifold with boundary $\partial M$. Then the following are equivalent:
 \begin{enumerate}
 \item $M$ is $\mathcal{D}$-parabolic;
 \item For every domain $\Omega\subset M$ and every bounded function $h\in C^{\infty}(int\; M)\cap C^0(M)$ satisfying $\Delta ^{M}h\geq 0$ on $int\Omega$ we have
 $$\sup _{\Omega}h=\sup _{\partial \Omega}h;$$
 \item For every bounded function $h$ satisfying $\Delta ^{M}h\geq 0$ on $int\;M$ we have 
 $$\sup _Mh=\sup _{\partial M}h.$$
 \end{enumerate}
 \end{prop}
 
 Now with the Proposition \ref{Dparabolica} and with previous discussion we can suppose in Theorem \ref{principio maximo} $M_2$ an $\mathcal{N}$-parabolic or $\mathcal{D}$-parabolic hypersurface and guarantee the validity of {\it Maximum Principles at Infinity}, namely we have the next theorem
 
 \begin{teo}\label{Nprincipio maximo}
 Let $M_1$ and $M_2$ two propperly embedded and disjoints hypersurfaces in $\mathbb{R}^{n+1}$ with nonempty boundaries $\partial M_1$ and $\partial M_2$ . Suppose that $M_2$ it is complete and that
 \begin{equation}
 \sup _{M_2}\vert\mathbf{H}_{M_2}\vert \leq b_0\leq\inf _{M_1}H_{M_1}\label{hip1},\ \ b_0>0.
 \end{equation}
 If $M_2$ have an ideal contac at infinity with $M_1$ and $M_2$ is $\mathcal{N}$-parabolic (or $\mathcal{D}$-parabolic) then
 \begin{equation}
 \min\{dist(M_1,\partial M_2),dist(M_2,\partial M_1)\}=0.
 \end{equation}
 \end{teo}
 
 For $\mathcal{D}$-parabolic hypersurfaces we have the following corollary of Theorem \ref{Nprincipio maximo}.
 
 \begin{corol}
 Let $M_1$ and $M_2$ has in Theorem \ref{Nprincipio maximo}. Assume that $M_2$ it is complete and 
 \begin{equation}
 \sup _{M_2}\vert\mathbf{H}_{M_2}\vert \leq b_0\leq\inf _{M_1}H_{M_1},\ \ b_0>0.
 \end{equation}
 Suppose that there exist relatively compact sets $\Omega _1\subset M_1$ and $\Omega _2\subset M_2$ such that $M_1\diagdown\Omega _1$ is isometric to $M_2\diagdown\Omega _2$. If $M_2$ have an ideal contact at infinity with $M_1$ and $M_1$ or $M_2$ is $\mathcal{D}$-parabolic then
 \begin{equation}
 \min\{dist(M_1,\partial M_2),dist(M_2,\partial M_1)\}=0.
 \end{equation}
 \end{corol}
 \begin{proof}
 By Corollary 13 of \cite{leandro} we have that if $M_1\diagdown\Omega _1$ is isometric to $M_2\diagdown\Omega _2$, then $M_1$ is $\mathcal{D}$-parabolic if and only if so is $M_2$. Therefore, if $M_2$ is $\mathcal{D}$-parabolic by Theorem \ref{Nprincipio maximo} the lemma is true. If $M_1$ is $\mathcal{D}$-parabolic then $M_2$ so is, and again the lemma is true.
 \end{proof}
 
 
\begin{obs}
Since the Proposition \ref{prop1} and the Ahlfors Maximum Principles, for $\mathcal{N}$-parabolicity, is valid when $\partial M_2=\emptyset$ we have that the Theorems \ref{principio maximo} and \ref{Nprincipio maximo} remains valid if $\partial M_2=\emptyset$, in this case we have that (\ref{tese}) remain $dist(M_2,\partial M_1)=0$.
\end{obs}


\begin{obs}
If we suppose that $M_{2}$ is a compact hypersurface, we can withdraw the hypothesis of parabolicity in Theorem \ref{principio maximo}, using now the Divergence Theorem in its demonstration.
\end{obs}

\section{Geometric Applications}


Let $M\subset\mathbb{R}^{n+1}$ be a complete hypersurface, propperly embedded and empty boundary. We say that $M$ is convex with respect to the unit normal $\eta$ if its mean curvature function is positive. Observe that $M$ splits $\mathbb{R}^{n+1}$ in two connected component. We will define as \textit{convex side} of $M$ the component of $\mathbb{R}^{n+1}$ to which the mean curvature vector points at.


\subsection{Applications of the Maximum Principles at Infinty}

As a consequence of the \hyperlink{PMI}{proof} of the Theorem \ref{principio maximo} we have the following theorem that extends to $\mathbb{R}^{n+1}$, without assumptions about the Gaussian curvature, the Corollary 1 in \cite{RFL} and the Theorem 3.4 in \cite{R-M}, with a similar demonstration. It also extends, in the parabolic case, the Theorem 1 in \cite{ros-rosenberg}.
 

\begin{teo}\label{convexo}
Let $M_{1}$ and $M_{2}$ be two propperly embedded and disjoints hypersurfaces in $\mathbb{R}^{n+1}$ with empty boundaries. Suppose that $M_{2}$ is complete and that
  $$ \sup _{M_2}\vert{\bf H}_{M_2}\vert \leq b_0\leq
  \inf _{M_1}H_{M_1}, \; b_0>0. $$
  If $M_{2}$ is parabolic, it cannot lie in the convex side of $M_{1}$.
\end{teo}

\begin{proof}
Suppose that $M_{2}$ is in the convex side of $M_{1}$. If $dist(M_{1},M_{2})=0$ then $M_{1}$ and $M_{2}$ have an ideal contact at infinty. In this case we can proceed with the demonstration of Theorem \ref{principio maximo} defining for an  $\epsilon >0$ sufficiently close to zero, the sets


$$M_2(\epsilon)=\{x\in M_2 : dist(x,M_1)\leq \epsilon\}.$$
For each $x\in M_2(\epsilon)$, consider the sets
 $$S_x=\{y\in M_1: \vert y - x\vert=dist(x,M_1) \ \ e \ \ x-y=\lambda{\bf H}_{M_1}(y), \ \ \lambda> 0 \}$$
and $M'_2(\epsilon)\subset M_2(\epsilon)$ given by  $$M'_2(\epsilon)= \{x\in M_2(\epsilon): S_x\neq \emptyset\}.$$ 
If $C_2(\epsilon)\subset M'_2(\epsilon)$ is a connected component of $M'_2(\epsilon)$, we will have that $$\partial C_2(\epsilon)= \{x\in C_2(\epsilon): dist(x,M_1)=\epsilon\}.$$ because we are assuming that $\partial M_2=\emptyset$.
And finally, defining in $C_2(\epsilon)$ the function $\phi (x)=e^{-C_0d(x)}$, where $d(x)=dist(x,M_{1})$, we will get to the same contradition of Theorem \ref{principio maximo}, because we are assuming that $M_{2}$ is parabolic. So, $M_{2}$ cannot lie in the convex side of $M_{1}$ if $dist(M_1,M_2)=0$.


If  $dist(M_1,M_2)>0$, there are sequences $y_n\in M_1$ and $x_n\in M_2$ in a way that the sequence $y_n - x_n$ have a subsequence that converges to a vector $v\in \mathbb{R}^{n+1}$ with  $\vert v\vert =dist(M_1,M_2)$. So let $\overline{M}_2=M_2+v$. In this way, we have that $dist(M_1,\overline{M}_2)=0$ and that $M_1\cap\overline{M}_2\neq\emptyset$, otherwise $\overline{M}_2$ have an ideal contact at infinity with $M_1$, what cannot happen by previous paragraph. Let  $p\in M_1\cap\overline{M}_2$, as $\vert v\vert =dist(M_1,M_2)$ and $M_2$ is in the convex side of $M_{1}$, we have that $M_{1}$ and $\overline{M}_2$ have an ideal contact at $p$. By Hopf'S Maximum Principle, $M_{1}$ and $\overline{M}_2$ coincide in a neighbourhood of $p$. That means that $M_{1}$ differs from $M_{2}$ by a translation in $\mathbb{R}^{n+1}$ of length $\vert v\vert =dist(M_1,M_2)$. As the line segment that starts in $M_{1}$ and ends in $M_{2}$, whose length is $dist(M_1,M_2)$, is orthogonal to both $M_{1}$ and $M_{2}$, we have that in this neighborhood of $p$, $M_{1}$ and $M_{2}$ are parallels hyperplans, contradicting the hyptothesis that the mean curvature function of $M_{1}$ is postive. Then, if $dist(M_{1}, M_{2})>0$, $M_{2}$ cannot lie in the convex side of $M_{1}$, which proves the Theorem.
\end{proof}

\begin{obs}
Generally, an hypersurface $M_1\subset\mathbb{R}^{n+k}$ with principal curvatures $\lambda _1\leq\cdots\leq\lambda _n\leq \cdots\leq\lambda _{n+k-1}$ with respect to the unit normal $\eta$ is called $n$-convex mean with respect to $\eta$ if $\lambda _1+\cdots+\lambda _n\geq 0$. Then the demonstration of Theorem \ref{convexo} above actually give as the following theorem, with a similar demonstration.
\end{obs}


\begin{teo}\label{convexo1.2}
Let $M_{1}$ be an hypersurface of $\mathbb{R}^{n+k}$ whose mean curvature function is  postive and $M_{2}$ an $n$-dimensional $C^{2}$-submanifold of $\mathbb{R}^{n+k}$. Also let $M_{1}$ and $M_{2}$ be disjoints, propperly embedded in $\mathbb{R}^{n+k}$ with empty boundaries. Suppose that $M_{2}$ is complete and that
 $$
 \sup _{M_2}\vert\mathbf{H}_{M_2}\vert \leq\inf _{M_1} \Uplambda _n.
 $$
 If $M_{2}$ is parabolic, it cannot lie in the convex side of $M_{1}$.
\end{teo}


This last theorem is the version with an ideal contact at infinity of the Theorem 4 in \cite{U-S} in the case when $M_{2}$ is parabolic and it is called there \emph{A barrier principle for submanifolds of arbitrary codimension and bounded mean curvature}, see also in \cite{L-T} the case for minimal submanifolds.

\subsection{Applications of Omori-Yau Maximum Principles}

Adding an appropriate hypothesis about the Ricci curvature of $M_2$ hypersurface, we can withdraw its condition of parabolicity on Theorem \ref{convexo} and obtain an analog result given in Theorem \ref{dist} soon. Before, we will demonstrate the Lemma \ref{shapeoperator} below.
  
  \begin{lema}\label{shapeoperator}
  Let $M\subset \mathbb{R}^{n+k}$ be a properly embedded hypersurface and $A$ its shape operator. We denote $d_j=d(x_j)=dist(x_j,M)$, where $x_j\in\mathbb{R}^{n+k}$ it is a sequence of points such that $\lim_{j\to  +\infty}d_j=0$, and $A_j$ as the shape operator of $M$ in $y_j\in M$, such that $\vert x_j-y_j\vert=d_j$. Suppose that $$\limsup_{j\to +\infty}\| A_j\| <+\infty.$$
  If $II_{d_j}$ is the second fundamental form of the parallel hypersurface $M_{d_j}=d^{-1}(d_j)$ in $T_{x_j}M_{d_j}$, then $$\lim _{j\to +\infty}\| II_{d_j}\| <+\infty.$$
  \end{lema}
  \begin{proof}
  If $A_{d_j}$ is the shape operator of $M_{d_j}$ then we have the Riccati's equation, see \cite{Eschenburg},
  \begin{equation}\label{riccati}
  \overline{\nabla}_{Dd}A_{d_j}+A^2_{d_j}+\overline{R}(\cdot \; ,Dd)Dd=0
  \end{equation}
  where $\overline{\nabla}$ and $\overline{R}$ are the Riemannian connection and the tensor curvature of the $\mathbb{R}^{n+k}$, respectively. For $x_j\in \mathbb{R}^{n+k}$ such that $d_j\ll 1$ consider the normalized geodesic minimizer $\beta : [0,d_j]\rightarrow\mathbb{R}^{n+k}$ with $\beta (0)=y_j\in M$ and $\beta (d_j)=x_j\in M_{d_j}$. As $\beta$ is a line segment joining $y_j$ to $x_j$ and $d_j\ll 1$, then $\beta '(d)=Dd$ if $d\in [0,d_j]$. Let $v\in T_{x_j}M_{d_j}$ and $V$ its parallel transport alongside $\beta$ with $V(0)=v_0\in T_{y_j}M$. Denoting by $'$ the derivative alongside $\beta$ we have
  $$\langle A_{d}(V),V\rangle '=\langle \overline{\nabla}_{\gamma '}A_{d}(V),V\rangle =\langle \overline{\nabla}_{Dd}A_{d}(V),V\rangle.$$
  Thus, by (\ref{riccati})
  $$\langle A_{d}(V),V\rangle '=-\langle A^2_{d}(V),V\rangle -\langle \overline{R}(V,Dd)Dd,V\rangle .$$
  Therefore in $d=0$ we have
  $$\langle A_{d}(V),V\rangle '\mid _{d=0}=-\langle A^2_{j}(v_0),v_0\rangle -\langle \overline{R}(v_0,\eta)\eta,v_0\rangle.$$
  Thus, the Taylor expansion of $II_{d_j}(v)=\langle A_{d_j}(v),v\rangle$ around second fundamental form $II_{j}=II_{d_j}\mid_{d=0}$ of $M$ it is given by 
  \begin{equation}\label{taylor}
  \langle A_{d_j}(v),v\rangle=\langle A_{j}(v_0),v_0\rangle -[\langle A^2_{j}(v_0),v_0\rangle +\langle \overline{R}(v_0,\eta)\eta,v_0\rangle]d +O(d^2)
  \end{equation}
  Being the sectional curvature of $\mathbb{R}^{n+k}$ null, we have from (\ref{taylor}) that
  \begin{equation}\label{A_j}
  \vert II_{d_j}(v)\vert\leq\|A_j\|\vert v_0\vert ^2+\|A_j\|^2\vert v_0\vert ^2d + O(d^2).
  \end{equation}
  Therefore from (\ref{A_j})
  \begin{equation}\label{A_j2}
    \|II_{d_j}\|\leq\|A_j\|+\|A_j\|^2d + O(d^2).
    \end{equation}
    As $j\to +\infty$ in (\ref{A_j2}) we have what we wanted.
  \end{proof}
  
\begin{teo}\label{dist}
Let $M_1$ and $M_2$ be two disjoints, without boundary and properly embedded hypersurfaces in $\mathbb{R}^{n+1}$. Assume that $M_2$ is complete with Ricci curvature satisfying
\begin{center}
 $Ric_{M_2}\geq -(n-1)R_0$,\ \ $R_0>0.$
 \end{center} 
 Let $d(x)=dist(x,M_1)$ and $A_j$ as in Lemma \ref{shapeoperator} and suppose that 
 \begin{equation}
 \limsup_{j\to +\infty}\| A_j\| <+\infty
 \end{equation}
 for every sequence $x_j\in M_2$ such that $\lim _{j\to +\infty}d(x_j)=\inf_{M_2}dist(x,M_1)$. If 
 \begin{equation}\label{equacao10}
  \sup _{M_2}\vert {\bf H}_{M_2}\vert <\inf _{M_1}H_{M_1},
  \end{equation}
  then $M_2$ it connot lie in the convex side of $M_1$.
\end{teo}
\begin{proof}
By Omori-Yau's Maximum Principle, with the version for the minimum, there is a sequence $x_j\in M_2$ such that $\lim _{j\to\infty}d(x_j)=\inf_{M_2}dist(x,M_1)$ and we have that
\begin{equation}\label{equacao12}
\vert \nabla ^{M_2}d\vert (x_j) < \dfrac{1}{j}\ \ \ \ \ \   e \ \ \ \ \  \ \Delta ^{M_2}d(x_j) > -\dfrac{1}{j}, \ \ \ \ \forall j\in \mathbb{N}.
\end{equation}
Suppose that $M_2$ lies on the mean convex side of $M_1$ and that $dist(M_1,M_2)=0$. Then $M_2$ have an ideal contact at infinity with $M_1$. This way, for $j\in \mathbb{N}$ sufficiently larg, the  Lemma \ref{lema2.1} give us that
\begin{equation}\label{equacao13}
\Delta ^{M_2}d(x_j) - n\vert{\bf H}_{M_2}\vert + \displaystyle\sum_{i}^{n}\lambda_i(d_j) - \displaystyle\dfrac{II_{d_j}((\nabla ^{M_2}d)^T,(\nabla ^{M_2}d)^T)}{1-\vert\nabla^{M_2}d\vert ^2}(x_j)\leq 0
\end{equation}
where $\lambda_i(d_j)$ are the main curvatures of the parallel hypersurface $d^{-1}(d_j)$ in the orientation given by $\eta$ and $d_j=d(x_j)$. As we have that $\Delta ^{M_2}d(x_j) > -\dfrac{1}{j}$ then (\ref{equacao13}) give us that
\begin{equation}\label{equacao14}
\dfrac{1}{j}> - n\vert{\bf H}_{M_2}\vert + \displaystyle\sum_{i}^{n}\lambda_i(d_j) - \displaystyle\dfrac{II_{d_j}((\nabla ^{M_2}d)^T,(\nabla ^{M_2}d)^T)}{1-\vert\nabla^{M_2}d\vert ^2}.
\end{equation}
And from $\vert \nabla ^{M_2}d\vert (x_j) < \dfrac{1}{j}$ we have that $$\displaystyle\dfrac{\vert II_{d_j}((\nabla ^{M_2}d)^T,(\nabla ^{M_2}d)^T)\vert }{1-\vert\nabla^{M_2}d\vert ^2} \leq \| II_{d_j}\|\dfrac{\vert\nabla ^{M_2}d\vert ^2}{1-\vert\nabla^{M_2}d\vert ^2}<\| II_{d_j}\|\dfrac{1}{j^2-1}.$$
Then, from (\ref{equacao14})
\begin{equation}\label{equacao15}
\begin{array}{rll}
\| II_{d_j}\|\dfrac{1}{j^2-1}+\dfrac{1}{j}&>& \displaystyle\sum_{i}^{n}\lambda_i(d_j)-n\vert{\bf H}_{M_2}\vert\\
&\geq&\displaystyle\sum_{i}^{n}\lambda_i -n\vert{\bf H}_{M_2}\vert.
\end{array} 
\end{equation}
As $j\to +\infty$ in (\ref{equacao15}), we have, by Lemma \ref{shapeoperator}, that $$n\vert{\bf H}_{M_2}\vert\geq \lambda_1 + \cdots+\lambda_n$$ contradicting (\ref{equacao10}). Thus, if $dist(M_1,M_2)=0$, $M_2$ cannot lie in the convex side of $M_1$.

If $dist (M_1,M_2)>0$, let $\overline{M}_2$ has in second paragraph in the demonstration of Theorem \ref{convexo}. Again $M_1\cap\overline{M}_2\neq\emptyset$, otherwise $\overline{M}_2$ have an ideal contact at infinity with $M_1$, what cannot happen by previous paragraph. Let $p_0\in M_1\cap\overline{M}_2$, then $d$ would attain a minimum at $p_0$. In this case $\nabla ^{\overline{M}_2}d(p_0)=0$ and by Lemma \ref{lema2.2} we have $\Delta ^{\overline{M}_2}d(p_0)<0$, which contradicts that $p_0$ is a minimum point for $d$. Therefore, $M_2$ cannot lie in the convex side of $M_1$ if $dist(M_1,M_2)>0$ and ends the proof of theorem.
\end{proof}


\begin{obs}\label{obs4.2}
Let $M$ be an oriented hypersurface isometrically imersed in Riemannian space form $\mathbb{M}^{n+1}_c$. If $R_M$ is the normalized scalar curvature of $M$ and $A$ its shape operator, then from Gauss equation we have the following relationship, see for instance \cite{Alias},
$$\|A\|^2=n^2H_M^2-n(n-1)(R_M-c).$$
Therefore, if $R_M\geq c$ we have to $\|A\|^2\leq n^2H_M^2$, and so $\|A\|\leq nH_M$, if $H_M>0$. Then if the mean curvature of $M$ is bounded will $$\sup _M\|A\|\leq n\sup _MH_M<+\infty.$$
\end{obs}

With the Remark \ref{obs4.2}, the Theorem \ref{dist} give us the following corollary.
 
 \begin{corol}\label{convexo1.3}
Let $M_1$ and $M_2$ as in Theorem \ref{dist}. Assume that $M_2$ is complete with Ricci curvature bounded below. Let $R_{M_1}$ the scalar curvature (normalized) of $M_1$ and suppose that  $R_{M_1}\geq 0$. Suppose also that $H_{M_1}$ is bounded and that 
$$ \sup _{M_2}\vert {\bf H}_{M_2}\vert <\inf _{M_1}H_{M_1}.$$
Then $M_2$ cannot lie in the convex side of $M_1$.
\end{corol}
\begin{proof}
Suppose $M_2$ lies in the convex side of $M_1$ and that $dist(M_1,M_2)=0$. Let $II_{d_j}$ be the second fundamental form of the parallel hypersurface $d^{-1}(d_j)$, where $d_j=d(x_j)=dist(x_j,M_1)$ and $x_j\in M_2$ is a sequence of points given by Omori-Yau, i.e., with $d_j\to 0$ satisfying (\ref{equacao12}). Then accordingly with the Remark \ref{obs4.2} and by Lemma \ref{shapeoperator}, we have  $\lim_{j\to +\infty} \|II_{d_j}\|<+\infty$. Proceeding like the demonstration of Theorem \ref{dist} we have the desired.
\end{proof}

Notice that in the demonstration of Theorem \ref{dist} the hypotheses about $M_2$ are essentially so that we can use the Omori-Yau maximum principle on function $d(x)=dist(x,M_1)$. Thus, we can suppose others hypotheses on $M_2$ allowing us to use this principle. This can be done using the  Pigola-Rigoli-Setti's Theorem, see \cite{Pi-Ri-SE}, which extends the class of Riemannian manifolds for which hold the Omori-Yau maximum principle. Before, however, we give the following definition given by Pigola-Rigoli-Setti also in \cite{Pi-Ri-SE}.

\begin{defi} The Omori-Yau maximum principles is said to hold on Riemannian manifold $M$ if for any given $u\in C^2(M)$ with $u^*=\sup _{M}u<+\infty$, there exist a sequence of points $x_j\in M$, depending on $M$ and on $u$, such that
$$\lim _{j\to +\infty}u(x_j)=u^*, \ \ \vert\nabla ^M u\vert (x_j)<\dfrac{1}{j}, \ \ \Delta ^M u(x_j)<\dfrac{1}{j}.
$$
\end{defi}
\begin{teo}[Pigola-Rigoli-Setti]\label{Pi-Ri-Se}
Let $M$ a Riemannian manifold and assume there exists a non-negative function $\gamma$ satisfying the following:\newline

C1) $\gamma (x)\to +\infty$ as $x \to +\infty$;\newline

C2) $\exists B>0$ such that $\vert\nabla ^M\gamma\vert\leq B\sqrt{\gamma}$ off a compact set;\newline

C3) $\exists C>0$ such that $\Delta ^M\gamma\leq C\sqrt{\gamma G(\sqrt{\gamma})}$ off a compact set, were\newline $G:[0,+\infty )\rightarrow [0,+\infty )$ is a smooth function satisfying
$$G(0)>0, \ \ G'(0)\geq 0, \ \ \displaystyle\int_{0}^{+\infty}\dfrac{ds}{\sqrt{G(s)}}=+\infty, \ \ \displaystyle\limsup_{t\to +\infty}\dfrac{tG(\sqrt{t})}{G(t)}<+\infty.$$
Then the Omori-Yau maximum principle holds on $M$.

\end{teo}

Because of the Theorem of Pigola-Rigoli-Setti, Theorem \ref{Pi-Ri-Se} above, the Corollary \ref{convexo1.3} can be extended to a more general class of Riemannian manifold that satisfy the Omori-Yau Maximum Principle. Then, we have the

\begin{teo}
Let $M_1$ be without boundary and properly embedded hypersurface in $\mathbb{R}^{n+k}$ with bounded scalar curvature and whose mean curvature satisfies $b_0\leq H_{M_1}\leq b_1$, where $b_0,b_1>0$. Let also $M_2$ be without boundary, disjoint of $M_1$ and properly embedded $n$-dimensional submanifold of $\mathbb{R}^{n+k}$. Suppose there exist a non-negative function in $M_2$ satisfying the conditions $C1)$, $C2)$ and $C3)$ of Theorem \ref{Pi-Ri-Se} and 
$$\sup _{M_2}\vert \mathbf{H}_{M_2}\vert <\inf _{M_1}\Uplambda _n.$$
Then $M_2$ it cannot lie in the convex side of $M_1$.
\end{teo}

With analogous proof of Theorem \ref{dist} and extend in the case of {\it Ideal Contact at Infinity} Theorem 4 in \cite{U-S}. Extend also to submanifolds of arbitrary codimension $M_2$ the Theorem 3.4 in \cite{R-M}, Theorem 1 in \cite{ros-rosenberg} and Corollary 1 in \cite{RFL}.



\end{document}